\documentclass[12pt]{article}
\usepackage{verbatim}
\usepackage{latexsym}
\usepackage{amsfonts}
\usepackage{amsmath}
\usepackage{amsthm}
\usepackage{tikz,subcaption}

\usetikzlibrary{calc,intersections, through,arrows,decorations.markings}

\newcommand{\east}[2]{\draw[->] (#1+0.1,#2) -- +(0.8,0);}
\newcommand{\seast}[2]{\draw[->] (#1+0.05,#2-0.05) -- +(0.9,-0.9);}
\newcommand{\neast}[2]{\draw[->] (#1+0.05,#2+0.05) -- +(0.9,0.9);}
\newcommand{\deast}[2]{\draw(#1+0.1,#2) -- +(0.5,0);\draw[densely dashed] (#1+0.6,#2) -- +(0.8,0);}
\newcommand{\ddeast}[2]{\draw[densely dashed] (#1-0.3,#2) -- +(0.6,0);\draw[->] (#1+0.3,#2) -- +(0.6,0);}
\newcommand{\dseast}[2]{\draw(#1+0.05,#2-0.05) -- +(0.6,-0.6);\draw[densely dashed] (#1+0.5,#2-0.5) -- +(0.5,-0.5) -- +(0.9,-0.5);}
\newcommand{\dneast}[2]{\draw(#1+0.05,#2+0.05) -- +(0.6,0.6);\draw[densely dashed] (#1+0.5,#2+0.5) -- +(0.5,0.5) -- +(0.9,0.5);}

\newcommand{\permrect}[2]{\filldraw[fill opacity=0.6] (#1+2.5,#2+1.5)--(#1+2.5,#2-1.5)--(#1-2.5,#2-1.5)--(#1-2.5,#2+1.5)--cycle;
\draw (#1,#2) node [vertex]{};}
\newcommand{\permrectangle}[2]{
\permrect{#1}{#2}
\permrect{#1+20}{#2}
\permrect{#1}{#2+20}
\permrect{#1+20}{#2+20}
\permrect{#1-20}{#2}
\permrect{#1-20}{#2-20}
\permrect{#1}{#2-20}
\permrect{#1}{#2-20}
\permrect{#1-20}{#2+20}
\permrect{#1+20}{#2-20}
}

\newtheorem{theorem}{Theorem}
\newtheorem{lemma}[theorem]{Lemma}

\begin{document}

\title{Permutations that separate close elements}
\author{Simon R. Blackburn\thanks{
Department of Mathematics,
Royal Holloway University of London,
Egham, Surrey TW20 0EZ, United Kingdom,
\texttt{s.blackburn@rhul.ac.uk}}}
\maketitle

\begin{abstract}
Let $n$ be a fixed integer with $n\geq 2$. For $i,j\in\mathbb{Z}_n$, define $||i,j||_n$ to be the distance between $i$ and $j$ when the elements of $\mathbb{Z}_n$ are written in a cycle. So  $||i,j||_n=\min\{(i-j)\bmod n,(j-i)\bmod n\}$. For positive integers $s$ and $k$, the permutation $\pi:\mathbb{Z}_n\rightarrow\mathbb{Z}_n$ is \emph{$(s,k)$-clash-free} if $||\pi(i),\pi(j)||_n\geq k$ whenever $||i,j||_n<s$ with $i\not=j$. So an $(s,k)$-clash-free permutation $\pi$ can be thought of as moving every close pair of elements of $\mathbb{Z}_n$ to a pair at large distance. More geometrically, the existence of an $(s,k)$-clash-free permutation is equivalent to the existence of a set of $n$ non-overlapping $s\times k$ rectangles on an $n\times n$ torus, whose centres have distinct integer $x$-coordinates and distinct integer $y$-coordinates.

For positive integers $n$ and $k$ with $k<n$, let $\sigma(n,k)$ be the largest value of $s$ such that an $(s,k)$-clash-free permutation on $\mathbb{Z}_n$ exists. In a recent paper, Mammoliti and Simpson conjectured that
\[
\lfloor (n-1)/k\rfloor-1\leq \sigma(n,k)\leq \lfloor (n-1)/k\rfloor
\]
for all integers $n$ and $k$ with $k<n$. The paper establishes this conjecture, by explicitly constructing an $(s,k)$-clash-free permutation on $\mathbb{Z}_n$ with $s=\lfloor (n-1)/k\rfloor-1$.
Indeed, this construction is used to establish a more general conjecture of Mammoliti and Simpson, where for some fixed integer $r$ we require every point on the torus to be contained in the interior of at most $r$ rectangles.
\end{abstract}

\newpage



\section{Introduction}
\label{sec:introduction}

Let $n$ be a fixed integer with $n\geq 2$. Define $||i,j||_n=\min\{(i-j)\bmod n,(j-i)\bmod n\}$ to be the distance between $i$ and $j$ when the elements of $\mathbb{Z}_n$ are written in a cycle. Let $s$ and $k$ be positive integers. For a permutation $\pi$ of $\mathbb{Z}_n$, we define an \emph{$(s,k)$-clash} to be a pair $(i,j)$ of distinct elements of $\mathbb{Z}_n$ such that $||i,j||_n<s$ and $||\pi(i),\pi(j)||_n<k$. A permutation is $\emph{$(s,k)$-clash-free}$ if it has no $(s,k)$-clashes. An $(s,k)$-clash-free permutation~$\pi$ can be thought of as being `chaotic' in the sense that pairs of distinct elements at distance less than $s$ are always mapped under $\pi$ to pairs at distance at least $k$. The notion of a clash-free permutation was introduced by Mammoliti and Simpson~\cite{MammolitiSimpson}.

If we draw a set of $s\times k$ rectangles on an $n\times n$ torus with centres at points $(i,\pi(i))$, we see that a permutation is $(s,k)$-clash-free if and only if the rectangles do not intersect. See Figure~\ref{fig:geometric} for an example of a $(5,3)$-clash-free permutation depicted in this way.

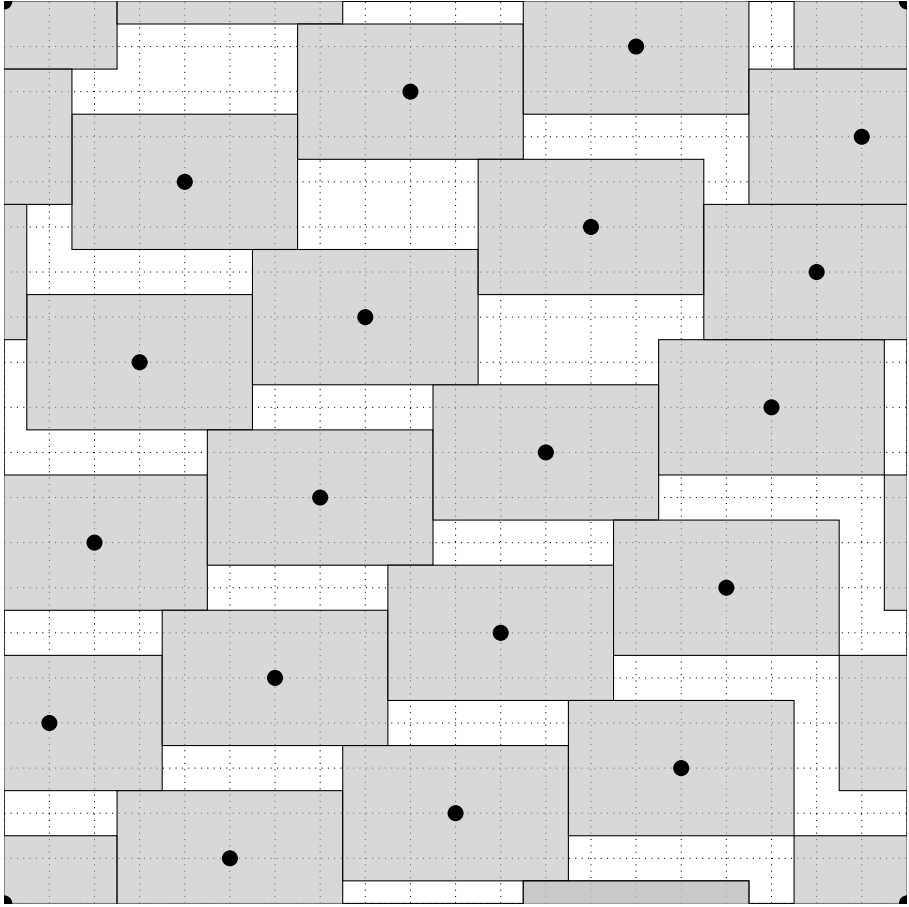
\begin{figure}
\begin{center}
\begin{tikzpicture}[fill=gray!50, scale=0.6,
vertex/.style={circle,inner sep=2,fill=black,draw}]

\clip (0,0) rectangle (20,20);

\draw[dotted] (0,0) grid (20,20);

\draw (0,0) rectangle (20,20);

\permrectangle{0}{0};
\permrectangle{1}{4};
\permrectangle{2}{8};
\permrectangle{3}{12};
\permrectangle{4}{16};
\permrectangle{5}{1};
\permrectangle{6}{5};
\permrectangle{7}{9};
\permrectangle{8}{13};
\permrectangle{9}{18};
\permrectangle{10}{2};
\permrectangle{11}{6};
\permrectangle{12}{10};
\permrectangle{13}{15};
\permrectangle{14}{19};
\permrectangle{15}{3};
\permrectangle{16}{7};
\permrectangle{17}{11};
\permrectangle{18}{14};
\permrectangle{19}{17};

\end{tikzpicture}
\end{center}
\caption{$5\times 3$ rectangles drawn on a $20\times 20$ torus, using the permutation $(\pi(0),\pi(1),\ldots ,\pi(19))$ of $\mathbb{Z}_{20}$ given by $(0,4,8,12,16,1,\ldots,11,14,17)$.}
\label{fig:geometric}
\end{figure}

We mention that `non-cyclic' variations of the problem, where we draw rectangles on a cylinder or a square, have also been considered~\cite{MammolitiSimpson}, as have generalisations of the $k=2$ case (known as cyclic matching sequencibility for graphs)~\cite{Alspach,BrualdiKiernan, KreherPastine}. 
We also mention papers by Bevan, Homberger and Tenner~\cite{BevanHomberger} and Blackburn, Homberger and Winkler~\cite{BlackburnHomberger}, which consider the related problem of packing diamonds rather than rectangles on a torus. Such packings are called \emph{permuted packings} in~\cite{BevanHomberger}. The problem of finding such packings can be rephrased as finding permutations $\pi$ of $\mathbb{Z}_n$ such that the distances between the points $(i,\pi(i))$ are large in the Manhattan metric (the $\ell_1$-metric). 

For integers $u$ and $v$ with $u\leq v$, we write $[u,v]$ for the set of all integers~$i$ such that $u\leq i\leq v$. We say that the \emph{length} of the interval is $v-u$. For integers $n$ and $k$, define $\sigma(n,k)$ to be the largest value of $s\in [1,n]$ such that an $(s,k)$-clash-free permutation of $\mathbb{Z}_n$ exists. It is easy to see directly from the definition of an $(s,k)$-clash that $\sigma(n,k)=1$ when $k\geq n$, so we may assume that $k<n$. We prove the following theorem:

\begin{theorem}
\label{thm:MSconjecture1}
Let $n$ and $k$ be integers with $1\leq k<n$. Define $\sigma(n,k)$ as above. Then
\[
\lfloor (n-1)/k\rfloor-1 \leq \sigma(n,k)\leq \lfloor (n-1)/k\rfloor.
\]
\end{theorem}

We prove this theorem in Section~\ref{sec:proof1} below. The theorem answers in the affirmative a conjecture of Mammoliti and Simpson~\cite[Conjecture~3.9]{MammolitiSimpson}. We remark that the two cases $\sigma(n,k)=\lfloor (n-1)/k\rfloor$ and $\sigma(n,k)=\lfloor (n-1)/k\rfloor-1$ both occur: A computer search in~\cite{MammolitiSimpson} showed that $\sigma(25,4)=6$ and $\sigma(26,4)=5$. These examples also show that $\sigma(n,k)$ is not monotonically increasing in $n$. The key contribution of this paper is to establish the lower bound in Theorem~\ref{thm:MSconjecture1} by explicitly constructing an $(s,k)$-clash-free permutation with $s=\lfloor (n-1)/k\rfloor-1$; the upper bound is~\cite[Lemma~2.3]{MammolitiSimpson}. 

The construction in the proof of Theorem~\ref{thm:MSconjecture1} can also be used in a generalisation (due to Mammoliti and Simpson) of the situation above, where we consider $r+1$-subsets rather than pairs. In more detail, for a subset $X\subseteq \mathbb{Z}_n$, we define $||X||_n$ to be the minimum length of an interval containing $X$. So, writing $x+[0,t]\bmod n$ for the subset $\{x,x+1,\ldots,x+t\}\subseteq\mathbb{Z}_n$, 
\[
||X||_n =\min\{t\in \mathbb{Z}\mid X\subseteq x+[0,t]\bmod n\text{ for some }x\in\mathbb{Z}_n\}.
\]
For a permutation $\pi$ of $\mathbb{Z}_n$, we define an \emph{$(s,k,r)$-clash} to be a subset $X\subseteq \mathbb{Z}_n$ of cardinality $r+1$ such that $||X||_n<s$ and $||\pi(X)||_n<k$. The permutation~$\pi$ is \emph{$(s,k,r)$-clash-free} if it has no $(s,k,r)$-clashes $X$. Geometrically, we are asking for a collection of $s\times k$ rectangles in the $n\times n$ torus with centres $(i,\pi(i))$ such that each point lies in the interior of no more than $r$ rectangles.  When $r=1$ we are in the situation we considered above.

For fixed integers $n$, $k$ and $r$, we define $\sigma(n,k,r)$ to be the largest value $s\in [0,n]$ such that an $(s,k,r)$-clash-free permutation of $\mathbb {Z}_n$ exists. When $r\geq n$ every permutation is $(s,k,r)$-clash-free, and so $\sigma(n,k,r)=n$. We have already dealt with the situation when $r=1$ above. So we may assume that $1<r<n$.  Similarly, when $k\geq n$ it is not hard to see that we have $\sigma(n,k,r)=r$, and when $r\geq k$ we have $\sigma(n,k,r)=n$ (the key fact for both situations being that every $r+1$-subset $X\subseteq\mathbb{Z}_n$ has $r\leq ||X||_n< n$). So we may assume that $1<r<k< n$.  We prove the following theorem in Section~\ref{sec:proof2} below:

\begin{theorem}
\label{thm:MSconjecture2}
Let $n$, $k$ and $r$ be integers such that $1<r<k< n$. Then
\[
\lfloor (rn-1)/k\rfloor-1 \leq \sigma(n,k,r)\leq \lfloor (rn-1)/k\rfloor.
\]
\end{theorem}
The theorem answers in the affirmative a conjecture of Mammoliti and Simpson~\cite[Conjecture~6.1]{MammolitiSimpson} (that generalises~\cite[Conjecture~3.9]{MammolitiSimpson}).

\section{A Proof of Theorem~\protect\ref{thm:MSconjecture1}}
\label{sec:proof1}

Mammoliti and Simpson~\cite[Lemma~2.3]{MammolitiSimpson} show that $\sigma(n,k)\leq \lfloor (n-1)/k\rfloor$. So it suffices to show that $\sigma(n,k)\geq \lfloor (n-1)/k\rfloor-1$.

Let $s=\lfloor (n-1)/k\rfloor-1$.
The theorem is trivially true when $s\leq 1$, as every permutation of $\mathbb{Z}_n$ is $(1,k)$-clash-free. So we may assume that $s\geq 2$. Similarly, since every permutation is $(s,1)$-clash-free, we may assume that $k\geq 2$.

To prove the theorem, it suffices to construct an $(s,k)$-clash-free permutation of $\mathbb{Z}_n$. In fact, we will construct a $(k,s)$-clash-free permutation $\pi$ of $\mathbb{Z}_n$. Since a pair $(i,j)$ is a $(k,s)$-clash for a permutation $\pi$ if and only if the pair $(\pi(i),\pi(j))$ is an $(s,k)$-clash for the permutation $\pi^{-1}$, we see that $\pi^{-1}$ is $(s,k)$-clash-free as required; see~\cite[Theorem~2.1]{MammolitiSimpson}.

Let $d=\gcd(s+1,n)$, and let $\ell$ be the additive order of $s+1$ in $\mathbb{Z}_n$. So $\ell = n/d$ and the additive group $\langle s+1\rangle$ generated by $s+1$ in $\mathbb{Z}_n$ consists of all elements $r\bmod n$ where $r$ is an integer such that $d$ divides $r$.

We see that $ks=k(\lfloor (n-1)/k\rfloor-1)\leq n-1-k$, and so
\begin{equation}
\label{eqn:ks}
k(s+1)\leq n-1.
\end{equation}
In particular, $i(s+1)\not\equiv0\bmod n$ for $i=1,2,\ldots ,k$, and hence $\ell> k$.

We define a $d\times \ell$ matrix $A$ over $\mathbb{Z}_n$ by setting the $(i,j)$ entry $A_{i,j}$ of $A$ to be $A_{i,j}=i+j(s+1)\bmod n$ for $0\leq i<d$ and $0\leq j<\ell$. An example of the matrix $A$ is given in Figure~\ref{fig:A}.

We will now show that every element of $\mathbb {Z}_n$ occurs exactly once in $A$. Row~$i$ of $A$ lists the $\ell$ elements of the coset $i+\langle s+1\rangle$. We claim these cosets are distinct. To see this, let $0\leq i<i'<d$. The difference between an element in row $i'$ and an element in row $i$ lies in the coset $(i'-i)+\langle s+1\rangle$. But $1\leq i'-i<d$, and so $i'-i$ is not a multiple of $d$. Hence $i'-i\notin\langle s+1\rangle$ and so the cosets listed in rows $i$ and $i'$ are distinct, as claimed. Thus the entries of the matrix $A$ are distinct. Since there are $d\ell=d(n/d)=n$ elements in $A$, every element of $\mathbb {Z}_n$ occurs exactly once in $A$, as required.

\begin{figure}
\[
\hspace{.37cm}
\left(
\begin{array}{ccccccccccccccccccc}
0&12&24&36&48&60&72&8&20&32&44&56&68&4&16&28&40&52&64\\
1&13&25&37&49&61&73&9&21&33&45&57&69&5&17&29&41&53&65\\
2&14&26&38&50&62&74&10&22&34&46&58&70&6&18&30&42&54&66\\
3&15&27&39&51&63&75&11&23&35&47&59&71&7&19&31&43&55&67
\end{array}
\right)
\]
\vspace{0.4cm}
\begin{tikzpicture}[fill=gray!50, scale=0.75]
\draw (-0.5,-0.5) -- (-0.5, 3.5);
\draw (0.5,-0.5) -- (0.5, 3.5);
\draw (1.5,-0.5) -- (1.5, 3.5);
\draw (2.5,-0.5) -- (2.5, 3.5);
\draw (3.5,-0.5) -- (3.5, 3.5);
\draw (4.5,-0.5) -- (4.5, 3.5);
\draw (5.5,-0.5) -- (5.5, 3.5);
\draw (6.5,-0.5) -- (6.5, 3.5);
\draw (7.5,-0.5) -- (7.5, 3.5);
\draw (8.5,-0.5) -- (8.5, 3.5);
\draw (9.5,-0.5) -- (9.5, 3.5);
\draw (10.5,-0.5) -- (10.5,3.5);
\draw (11.5,-0.5) -- (11.5, 3.5);
\draw (12.5,-0.5) -- (12.5, 3.5);
\draw (13.5,-0.5) -- (13.5, 3.5);
\draw (14.5,-0.5) -- (14.5, 3.5);
\draw (15.5,-0.5) -- (15.5, 3.5);
\draw (16.5,-0.5) -- (16.5, 3.5);
\draw (17.5,-0.5) -- (17.5, 3.5);
\draw (18.5,-0.5) -- (18.5, 3.5);

\draw (-0.5,-0.5) -- (18.5,-0.5);
\draw (-0.5,0.5) -- (18.5,0.5);
\draw (-0.5,1.5) -- (18.5,1.5);
\draw (-0.5,2.5) -- (18.5,2.5);
\draw (-0.5,3.5) -- (18.5,3.5);


\ddeast{-1}{3} \east{0}{3} \east{1}{3} \east{2}{3} \east{3}{3} \east{4}{3}\east{5}{3}\east{6}{3}\east{7}{3}\east{8}{3}\east{9}{3}\east{10}{3}\east{11}{3}\east{12}{3}\east{13}{3}\east{14}{3}\east{15}{3}\east{16}{3}\east{17}{3}\dseast{18}{3}
\ddeast{-1}{2} \east{0}{2} \east{1}{2} \east{2}{2} \east{3}{2}\east{4}{2}\east{5}{2}\east{6}{2}\east{7}{2}\east{8}{2}\east{9}{2}\east{10}{2}\east{11}{2}\east{12}{2}\east{13}{2}\east{14}{2}\east{15}{2}\east{16}{2}\seast{17}{2} \dneast{18}{2}
\ddeast{-1}{1} \east{0}{1} \east{1}{1} \east{2}{1}\east{3}{1}\east{4}{1}\east{5}{1}\east{6}{1}\east{7}{1}\east{8}{1}\east{9}{1}\east{10}{1}\east{11}{1}\east{12}{1}\east{13}{1}\east{14}{1}\east{15}{1}\seast{16}{1} \neast{17}{1} \deast{18}{1}
\ddeast{-1}{0} \east{0}{0} \east{1}{0}\east{2}{0} \east{3}{0} \east{4}{0} \east{5}{0} \east{6}{0} \east{7}{0} \east{8}{0} \east{9}{0} \east{10}{0} \east{11}{0} \east{12}{0} \east{13}{0} \east{14}{0} \east{15}{0} \neast{16}{0} \east{17}{0} \deast{18}{0}
\end{tikzpicture}
\caption{$A$ in the case $(n,k,s)=(76,6,11)$, and the cyclic ordering of the entries of $A$ we use. So $(\pi(0),\pi(1),\ldots,\pi(n-1))$ is defined as $(0,12,24,36,\ldots,52,64,1,13,25,\ldots,41,53,66,2,14,\ldots,19,31,43,54,65)$.}
\label{fig:A}
\end{figure}
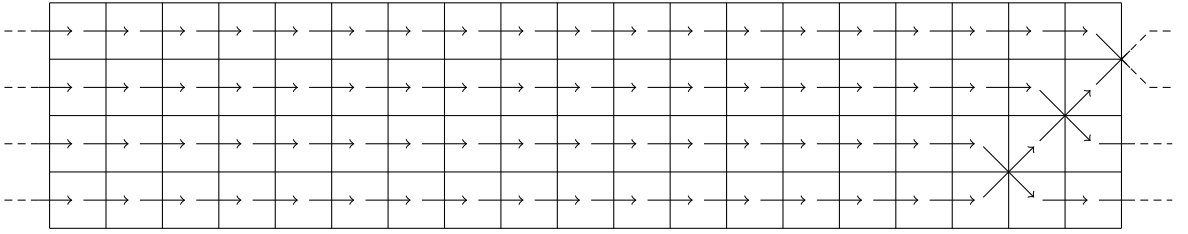

The following three properties of $A$ are easy to verify. First, moving one place east (cyclically), in $A$ always increases the entry by $s+1$:
\begin{equation}
\begin{split}
A_{i,(j+1)}-A_{i,j} &\equiv s+1\bmod n\text{ for }0\leq i<d\text{ and }0\leq j<\ell-1,\\
A_{i,0}-A_{i,(\ell-1)}&\equiv s+1\bmod n.
\end{split}\label{eqn:east}
\end{equation}
Secondly, moving southeast by one position (cylically), when not in the bottom row, increases the entry by $s+2$:
\begin{equation}
\begin{split}
A_{(i+1),(j+1)}-A_{i,j} &\equiv s+2\bmod n\text{ for }0\leq i<d-1\text{ and }0\leq j<\ell-1,\\
A_{(i+1),0}-A_{i,(\ell-1)}&\equiv s+2\bmod n\text{ for }0\leq i<d-1.
\end{split}\label{eqn:se}
\end{equation}
Finally, moving northeast by one position (cylically), when not in the top row, increases the entry by $s$:
\begin{equation}
\begin{split}
A_{(i-1),(j+1)}-A_{i,j} &\equiv s\bmod n\text{ for }1\leq i<d\text{ and }0\leq j<\ell-1,\\
A_{(i-1),0}-A_{i,(\ell-1)}&\equiv s\bmod n\text{ for }1\leq i<d.
\end{split}\label{eqn:ne}
\end{equation}

We order the entries of $A$ into a cycle as follows. We begin at the entry $A_{0,0}$. If we are currently at the entry $A_{i,j}$, we move east unless:
\begin{itemize}
\item $j\equiv (\ell-1)-i\bmod \ell$ and $i<d-1$, when we move south east;
\item $j\equiv (\ell-1)-i+1\bmod \ell$ and $i>0$, when we move north east.
\end{itemize}
Our cycle is well-defined, since $\ell>k\geq 2$. An example of this cycle is given in the lower part of Figure~\ref{fig:A}. When $d=1$ our array has a single row, and all our moves are east-moves. We note that southeast-moves are separated by at least $\ell-2$ east-moves (and the $d-1$ northeast-moves occur consecutively). Moreover the cycle visits all the entries of $A$. 

Define a function $\pi:\mathbb{Z}_n\rightarrow\mathbb{Z}_n$ by
\[
(\pi(0),\pi(1),\pi(2),\ldots,\pi(n-1))=(A_{0,0},A_{0,1},A_{0,2},\ldots,A_{1,\ell-1}),
\]
where the entries in $A$ on the right hand side are ordered using our cycle. Since every element of $\mathbb{Z}_n$ occurs in $A$, and our cycle visits all entries of $A$, we see that $\pi$ is a permutation.

We will now show that $\pi$ has no $(k,s)$-clashes. This is sufficient to establish the theorem. Let $z$ be a positive integer, with $z<k$. Let $i,j\in \mathbb{Z}_n$ be such that $||j-i||_n=z$. Without loss of generality, we may assume that the entry of $A$ corresponding to $\pi(j)$ can be reached by moving $z$ times along our cycle, starting at the entry corresponding to $\pi(i)$. Suppose these $z$ moves are made up of $a_0$ northeast-moves, $a_1$ east-moves and $a_2$ southeast moves, where $a_0+a_1+a_2=z$. The equations~\eqref{eqn:east},~\eqref{eqn:se} and~\eqref{eqn:ne} show that $\pi(j)\equiv\pi(i)+\delta\bmod n$, where $\delta=zs+a_1+2a_2$. Since $z$ is positive, $\delta\geq s$. Now $z<k$, and we showed earlier that  $k<\ell$. Hence $z\leq \ell-2$. Each southeast move is separated by at least $\ell-2$ east-moves on our cycle, and the $z$ moves we are examining are consecutive on our cycle, so we must have $a_2\leq 1$. Hence
\begin{align*}
\delta&= zs+a_1+2a_2\leq zs+z+1=z(s+1)+1\\
&\leq (k-1)(s+1)+1\\
&\leq (n-1)-(s+1)+1 \text{ by~\eqref{eqn:ks}}\\
&=n-1-s<n-s.
\end{align*}
But since $\pi(j)\equiv\pi(i)+\delta$ where $s\leq \delta\leq n-s$ we see that $||\pi(i),\pi(j)||_n\geq s$. We have shown that whenever $||j-i||_n=z<k$ we have $||\pi(i),\pi(j)||_n\geq s$. So $\pi$ is a $(k,s)$-clash-free permutation, and the theorem follows.~$\Box$

As an aside, we note that, in the notation of the proof above, the array $A$ is always wider than it is tall (in other words, $\ell>d$). This means that movements between rows are always at the right-hand side of the array. To prove this, first note that $n-k\leq k(s+1)\leq n-1$, as $s+1=\lfloor (n-1)/k\rfloor$. Write $k(s+1)=n-x$ with $1\leq x\leq k$. Since $d=\gcd(s+1,n)$, we see that $d$ divides $n$ and $k(s+1)=n-x$, and so $d$ divides $x$. In particular, $d\leq x\leq k$. But we have already shown, as part of the proof of Theorem~\ref{thm:MSconjecture1} above, that $k<\ell$ and so $d< \ell$ as required.

Theorem~\ref{thm:MSconjecture1} shows that 
\[
\sigma(n,k)\leq \lfloor (n-1)/k\rfloor-1+\varepsilon(n,k)
\]
for some integer $\varepsilon(n,k)\in\{0,1\}$ which depends on $n$ and $k$.  (The construction shows that $\varepsilon(n,k)\geq 0$.)
A very natural question is: For which values of $n$ and $k$ can we determine $\varepsilon(n,k)$? Mammoliti and Simpson show~\cite[Theorem~3.7]{MammolitiSimpson} that $\varepsilon(n,k)=1$ when $k|n$, when $\lfloor(n-1)/k\rfloor |n$, when $\gcd(n,k)=1$, or when $\gcd(n,\lfloor (n-1)/k\rfloor)=1$.  Moreover, they observe~\cite[Corollary~3.8]{MammolitiSimpson} that when there exist integers $k_1$ and $k_2$ such that:
\begin{gather*}
k_1<k_2<n\\
\lfloor(n-1)/k_1\rfloor=\lfloor(n-1)/k_2\rfloor, \text{ and}\\
\varepsilon(n,k_2)=1
\end{gather*}
then we can conclude that $\varepsilon(n,k_1)=1$. Using these results, together with computer searches in the cases $(n,k)=(18,4)$ and $(n,k)=(26,6)$, the authors determine $\varepsilon(n,k)$ for $n\leq 30$. Indeed, they show that $\varepsilon(18,4)=\varepsilon(26,4)=\varepsilon(26,6)=0$, and in all other cases with $n\leq 30$ we have $\varepsilon(n,k)=1$. Extending to $n\leq 40$, the open cases after using these results are $(n,k)\in\{(34,4),(34,8),(38,6), (39,6), (40,6)\}$.

In the justification for the construction in the proof of Theorem~\ref{thm:MSconjecture1}, the property of $s$ we need is that $(k-1)(s+1)\leq n-s$. So when
\[
(k-1)(\lfloor(n-1)/k\rfloor+1)+1\leq n-\lfloor(n-1)/k\rfloor
\]
we may set $s=\lfloor(n-1)/k\rfloor$ in the construction to produce a $(k,\lfloor(n-1)/k\rfloor)$-clash free permutation. This allows us to deduce that $\varepsilon(n,k)=1$ for such parameters. An example of this phenomenon is given in Figure~\ref{fig:B}. (The array $A$ is no longer necessarily wider than it is tall, as can be seen in this example.) However, it seems (at least for $n\leq 1000$) that no cases are resolved that are not already known from the results of Mammoliti and Simpson.

\begin{figure}
\centering
\begin{subfigure}{6cm}
\[
\left(
\begin{array}{ccccc}
0& 24& 48& 72& 96\\
 1& 25& 49& 73& 97\\
  2& 26& 50& 74& 98\\
   3& 27& 51&
75& 99\\
 4& 28& 52& 76& 100\\
  5& 29& 53& 77& 101\\
  6& 30
 & 54& 78& 102\\
7& 31& 55& 79& 103\\
 8& 32& 56& 80& 104\\
  9& 33& 57& 81& 105\\
   10& 34&58& 82& 106\\
    11& 35& 59& 83& 107\\
     12& 36& 60& 84& 108\\
 13& 37& 61&
85& 109\\
 14& 38& 62& 86& 110\\
  15& 39& 63& 87& 111\\
  16&
   40& 64& 88&
112\\
 17&41& 65& 89& 113\\
 18&
  42& 66& 90& 114\\
  19& 43& 67& 91& 115\\
20& 44& 68& 92& 116\\
21& 45& 69& 93& 117\\
22& 46& 70& 94& 118\\
23&47& 71& 95& 119
\end{array}
\right)
\]
\end{subfigure}
\begin{subfigure}{4cm}
\vspace{0.4cm}
\begin{tikzpicture}[fill=gray!50, scale=0.5]
\draw (-0.5,-0.5) -- (-0.5, 23.5);
\draw (0.5,-0.5) -- (0.5, 23.5);
\draw (1.5,-0.5) -- (1.5, 23.5);
\draw (2.5,-0.5) -- (2.5, 23.5);
\draw (3.5,-0.5) -- (3.5, 23.5);
\draw (4.5,-0.5) -- (4.5, 23.5);

\draw (-0.5,-0.5) -- (4.5,-0.5);
\draw (-0.5,0.5) -- (4.5,0.5);
\draw (-0.5,1.5) -- (4.5,1.5);
\draw (-0.5,2.5) -- (4.5,2.5);
\draw (-0.5,3.5) -- (4.5,3.5);
\draw (-0.5,4.5) -- (4.5,4.5);
\draw (-0.5,5.5) -- (4.5,5.5);
\draw (-0.5,6.5) -- (4.5,6.5);
\draw (-0.5,7.5) -- (4.5,7.5);
\draw (-0.5,8.5) -- (4.5,8.5);
\draw (-0.5,9.5) -- (4.5,9.5);
\draw (-0.5,10.5) -- (4.5,10.5);
\draw (-0.5,11.5) -- (4.5,11.5);
\draw (-0.5,12.5) -- (4.5,12.5);
\draw (-0.5,13.5) -- (4.5,13.5);
\draw (-0.5,14.5) -- (4.5,14.5);
\draw (-0.5,15.5) -- (4.5,15.5);
\draw (-0.5,16.5) -- (4.5,16.5);
\draw (-0.5,17.5) -- (4.5,17.5);
\draw (-0.5,18.5) -- (4.5,18.5);
\draw (-0.5,19.5) -- (4.5,19.5);
\draw (-0.5,20.5) -- (4.5,20.5);
\draw (-0.5,21.5) -- (4.5,21.5);
\draw (-0.5,22.5) -- (4.5,22.5);
\draw (-0.5,23.5) -- (4.5,23.5);


\ddeast{-1}{23} \east{0}{23} \east{1}{23} \east{2}{23} \east{3}{23} \dseast{4}{23}
\ddeast{-1}{22} \east{0}{22} \east{1}{22} \east{2}{22} \seast{3}{22} \dneast{4}{22}
\ddeast{-1}{21} \east{0}{21} \east{1}{21} \seast{2}{21} \neast{3}{21} \deast{4}{21}
\ddeast{-1}{20} \east{0}{20} \seast{1}{20} \neast{2}{20} \east{3}{20} \deast{4}{20}
\ddeast{-1}{19} \seast{0}{19} \neast{1}{19} \east{2}{19} \east{3}{19} \deast{4}{19}
\ddeast{-1}{18} \neast{0}{18} \east{1}{18} \east{2}{18} \east{3}{18} \dseast{4}{18}
\ddeast{-1}{17} \east{0}{17} \east{1}{17} \east{2}{17} \seast{3}{17} \dneast{4}{17}
\ddeast{-1}{16} \east{0}{16} \east{1}{16} \seast{2}{16} \neast{3}{16} \deast{4}{16}
\ddeast{-1}{15} \east{0}{15} \seast{1}{15} \neast{2}{15} \east{3}{15} \deast{4}{15}
\ddeast{-1}{14} \seast{0}{14} \neast{1}{14} \east{2}{14} \east{3}{14} \deast{4}{14}
\ddeast{-1}{13} \neast{0}{13} \east{1}{13} \east{2}{13} \east{3}{13} \dseast{4}{13}
\ddeast{-1}{12} \east{0}{12} \east{1}{12} \east{2}{12} \seast{3}{12} \dneast{4}{12}
\ddeast{-1}{11} \east{0}{11} \east{1}{11} \seast{2}{11} \neast{3}{11} \deast{4}{11}
\ddeast{-1}{10} \east{0}{10} \seast{1}{10} \neast{2}{10} \east{3}{10} \deast{4}{10}
\ddeast{-1}{9} \seast{0}{9} \neast{1}{9} \east{2}{9} \east{3}{9} \deast{4}{9}
\ddeast{-1}{8} \neast{0}{8} \east{1}{8} \east{2}{8} \east{3}{8} \dseast{4}{8}
\ddeast{-1}{7} \east{0}{7} \east{1}{7} \east{2}{7} \seast{3}{7} \dneast{4}{7}
\ddeast{-1}{6} \east{0}{6} \east{1}{6} \seast{2}{6} \neast{3}{6} \deast{4}{6}
\ddeast{-1}{5} \east{0}{5} \seast{1}{5} \neast{2}{5} \east{3}{5} \deast{4}{5}
\ddeast{-1}{4} \seast{0}{4} \neast{1}{4} \east{2}{4} \east{3}{4} \deast{4}{4}
\ddeast{-1}{3} \neast{0}{3} \east{1}{3} \east{2}{3} \east{3}{3} \dseast{4}{3}
\ddeast{-1}{2} \east{0}{2} \east{1}{2} \east{2}{2} \seast{3}{2} \dneast{4}{2}
\ddeast{-1}{1} \east{0}{1} \east{1}{1} \seast{2}{1} \neast{3}{1} \deast{4}{1}
\ddeast{-1}{0} \east{0}{0} \east{1}{0} \neast{2}{0} \east{3}{0} \deast{4}{0}

\end{tikzpicture}\end{subfigure}

\caption{$A$ in the case $(n,k,s)=(120,5,23)$, and the cyclic ordering of the entries of $A$ we use. So $(\pi(0),\pi(1),\ldots,\pi(n-1))$ is defined as $(0,24,48,72,96,1,25,49,73,98,2,26,50,75,99,3,27,\ldots,28,51,74,97)$.}
\label{fig:B}
\end{figure}
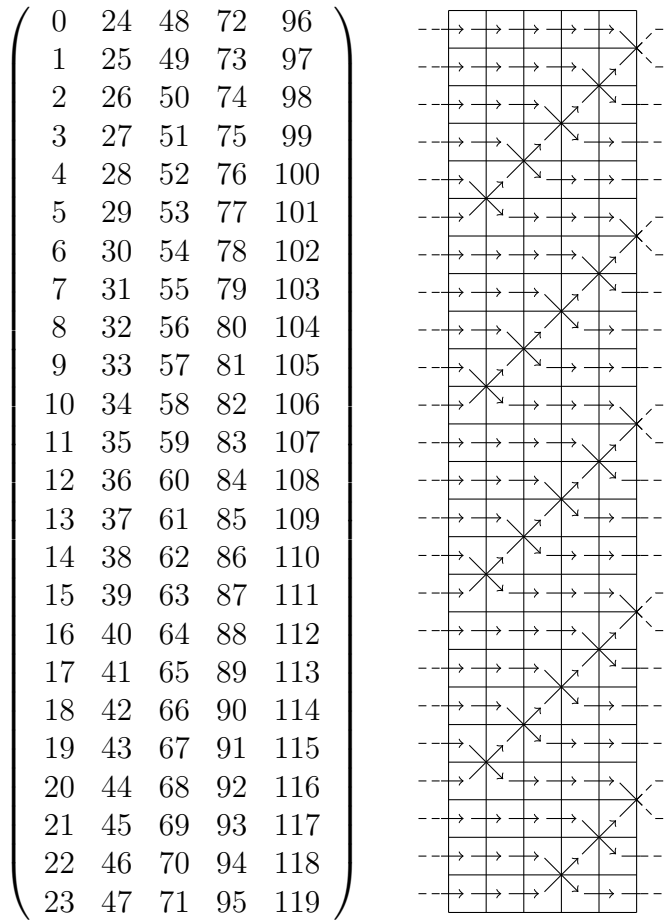

\section{A Proof of Theorem~\protect\ref{thm:MSconjecture2}}
\label{sec:proof2}

We use the following lemma, which follows from the fact that $X$ is a $(k,s,r)$-clash for a permutation $\pi$ if and only if $\pi(X)$ is an $(s,k,r)$-clash for $\pi^{-1}$.

\begin{lemma}
\label{lem:inverse}
A permutation $\pi$ of $\mathbb{Z}_n$ is $(k,s,r)$-clash-free if and only if $\pi^{-1}$ is $(s,k,r)$-clash-free.
\end{lemma}

We separate the proof of Theorem~\ref{thm:MSconjecture2} into two parts. We first prove Theorem~\ref{thm:construction} below, which constructs a permutation that is $(s,k,r)$-clash-free provided a certain inequality holds. We then prove Theorem~\ref{thm:MSconjecture2} by showing that this inequality holds in the situations we need.

\begin{theorem}
\label{thm:construction}
Let $r$, $s$, $k$ and $n$ be positive integers with $r<k< n$ and $s+1<n$. Let $d=\gcd(s+1,n)$ and let $\ell=n/d$. There exists an $(s,k,r)$-clash-free permutation of $\mathbb{Z}_n$ whenever
\begin{equation}
\label{eqn:construction_bound}
k(s+1)+d-3\leq rn-1.
\end{equation}
\end{theorem}

\begin{proof} By Lemma~\ref{lem:inverse}, just as in Section~\ref{sec:proof1}, we choose to construct a $(k,s,r)$-clash-free permutation $\pi$ of $\mathbb{Z}_n$.

Since $d\leq s+1<n$, we see that $\ell=n/d\geq 2$.
We construct the permutation $\pi$ exactly as in Section~\ref{sec:proof1}. For $i\in\mathbb{Z}_n$, we have $\pi(i+1)\equiv \pi(i)+\delta_i$, where $\delta_i\in\mathbb{Z}$ is equal to one of $s$, $s+1$ or $s+2$, depending on whether we have moved north-east, east, or southeast when moving from position $i$ to position $i+1$ in our cycle.

It remains to prove that $\pi$ is $(k,s,r$)-clash free. Suppose, for a contradiction, that $\pi$ has a $(k,s,r$)-clash. So there exists an $(r+1)$-subset $X\subset\mathbb{Z}_n$ such that $||X||_n<k$ and $||\pi(X)||_r<s$. Let $x\in\mathbb{Z}_n$ be such that $X\subseteq x+[0,k-1]\bmod n$, 
and let $y\in\mathbb{Z}_n$ be such that $\pi(X)\subseteq y+[0,s-1]\bmod n$. In particular, there exists an element of $\mathbb{Z}_n$  (namely the element $y+s-1$) that can be written in at least $r+1$ ways in the form $\pi(i)+j$ with $i\in X$ and $j\in [0,s-1]$. Since $X\subseteq x+[0,k-1]$, this element of $\mathbb{Z}_n$ can be written in at least $r+1$ ways in the form $\pi(x+i)+j$ with $i\in [0,k-1]$ and $j\in [0,s-1]$. We derive our contradiction by showing that this cannot happen.

Define integers $\sigma_0,\sigma_1,\ldots,\sigma_{k-1}$ by $\sigma_0=0$ and $\sigma_{i+1}=\sigma_i+\delta_{x+i}$ for $i\in\{0,1,\ldots,k-2\}$. So $\pi(x+i)= \pi(x)+\sigma_i\bmod n$. To prove the theorem, it is sufficient to show that the sets $\pi(x)+\sigma_i+[0,s-1]\bmod n$ (with $i=0,1,\ldots k-1$) cover each element of $\mathbb{Z}_n$ no more than $r$ times. This is equivalent to showing that the sets $\sigma_i+[0,s-1]\bmod n$ cover each element of $\mathbb{Z}_n$ no more than $r$ times.

The subsets $\sigma_i+[0,s-1]\subseteq\mathbb{Z}$ are pairwise disjoint, since $\delta_h\geq s$ for all $h\in\mathbb{Z}_n$. When moving from position $x$ to position $x+k-1$ along our cycle, we have moved northeast $a_0$ times, east $a_1$ times and southeast $a_2$ times for some non-negative integers $a_0$, $a_1$ and $a_2$ with $a_0+a_1+a_2=k-1$. Note that $k\leq n$ and we move southeast at most $d-1$ times in our cycle, and so $a_2-a_0\leq a_2\leq d-1$. Thus
\[
\sigma_{k-1}=a_0s+a_1(s+1)+a_2(s+2)=(k-1)(s+1)+a_2-a_0\leq (k-1)(s+1)+d-1.
\]
Thus the disjoint sets $\sigma_i+[0,s-1]$ are contained in the integer interval $[0,t]$ where
\begin{align*}
t&=\sigma_{k-1}+s-1\\
&\leq (k-1)(s+1)+d-1 +s-1\\
&=k(s+1)+d-3.
\end{align*}
Now $t\leq rn-1$, by~\eqref{eqn:construction_bound}, and the map from the integer interval $[0,rn-1]$ to $\mathbb{Z}_n$ given by $z\mapsto (z\bmod n)$ is $r$-to-one. Thus the sets $\sigma_i+[0,s-1]\bmod n$ cover each element of $\mathbb{Z}_n$ no more than $r$ times. So we have our contradiction, as required.
\end{proof}

\begin{proof}[Proof of Theorem~\ref{thm:MSconjecture2}]
We have $\sigma(n,k,r)\leq \lfloor (rn-1)/k\rfloor$ by~\cite[Lemma~4.1]{MammolitiSimpson}. So in order to prove Theorem~\ref{thm:MSconjecture2}, it suffices to show that an $(s,k,r)$-clash-free permutation of $\mathbb{Z}_n$ exists with $s=\lfloor (rn-1)/k\rfloor-1$. Define $d=\gcd(s+1,n)$ and $\ell=n/d$. 

By the quotient and remainder theorem, we may write $rn-1=qk+\rho$ where $q$ is a non-negative integer and $0\leq \rho<k$. We see that $s+1=\lfloor (rn-1)/k\rfloor=q$. Hence
\begin{align*}
k(s+1)+d-3 &=qk+\gcd(n,q)-3\\
&\leq qk+\gcd(rn,q)-3\\
&=qk+\gcd(qk+\rho+1,q)-3\\
&=qk+\gcd(\rho+1,q)-3\\
&\leq qk+\rho+1-3\\
&=rn-1-2\\
&<rn-1.
\end{align*}
The theorem follows, by Theorem~\ref{thm:construction}.
\end{proof}

We note that the construction here sometimes determines the value of $\sigma(n,k,r)$ exactly. For example, when $r=2$ and
\[
(n,k)\in\{(7,3),(10,3),(12,5),(13,3),(16,3),(17,5),(17,7),(19,3)\}
\]
the conditions of Theorem~\ref{thm:construction} are satisfied with $s=\lfloor (rn-1)/k\rfloor$, hence $\sigma(n,k,2)=\lfloor (rn-1)/k\rfloor$ in these cases. It would be very interesting if the exact value of $\sigma(n,k,r)$ could be determined in all cases.

\paragraph{Acknowledgements} The author is grateful to the reviewers for their useful and perceptive comments, which have significantly improved the paper.

\end{document}